\input ./style/arxiv-vmsta.cfg
\documentclass[numbers,compress,v1.0.1]{vmsta}
\usepackage{vtexbibtags}

\volume{5}
\issue{1}
\pubyear{2018}
\firstpage{37}
\lastpage{52}
\aid{VMSTA94}
\doi{10.15559/18-VMSTA94}
\articletype{research-article}

\startlocaldefs
\def\@journal@url{http://www.vmsta.org}
\def\@credit{%
  \vbox to 0pt{%
    \vskip-1.85pc
    \hskip-\textwidth
    \noindent
    \raise4mm\hbox to \textwidth{%
      \footnotesize
      \href{\@journal@url}{www.vmsta.org}%
      \hfill
      \href{\@vtex@url}{\includevtexlogo}%
      }%
    }%
  }

\newcommand{\rrvert}{\vert}
\newcommand{\llvert}{\vert}
\allowdisplaybreaks
\endlocaldefs

\newtheorem{thm}{Theorem}
\newtheorem{lem}{Lemma}
\newtheorem{cor}{Corollary}
\theoremstyle{definition}
\newtheorem{defin}{Definition}

\begin{document}

\begin{frontmatter}
\pretitle{Research Article}

\title{Confidence regions in Cox proportional hazards model with
measurement errors and unbounded parameter~set}

\author{\inits{O.}\fnms{Oksana} \snm{Chernova}\thanksref{cor1}\ead[label=e1]{chernovaoksan@gmail.com}}
\thankstext[type=corresp,id=cor1]{Corresponding author.}
\author{\inits{A.}\fnms{Alexander} \snm{Kukush}\ead[label=e1]{alexander\_kukush@univ.kiev.ua}}
\address{\institution{Taras Shevchenko National University of Kyiv},
Kyiv, \cny{Ukraine}}



\markboth{O. Chernova, A. Kukush}{Confidence regions in Cox
proportional hazards model}

\begin{abstract}
Cox proportional hazards model with measurement errors is
considered. In Kukush and Chernova (2017), we elaborated a
simultaneous estimator of the baseline hazard rate $\lambda(\cdot)$
and the regression parameter $\beta$, with the unbounded parameter set
$\varTheta=\varTheta_{\lambda}\times\varTheta_{\beta}$, where
$\varTheta_{\lambda}$ is a closed convex subset of $C[0,\tau]$ and
$\varTheta_{\beta}$ is a compact set in $\mathbb{R}^m$. The estimator
is consistent and asymptotically normal. In the present paper, we
construct confidence intervals for integral functionals of
$\lambda(\cdot)$ and a confidence region for $\beta$ under
restrictions on the error distribution. In particular, we handle
the following cases: (a) the measurement error is bounded, (b) it
is a normally distributed random vector, and (c) it has independent
components which are shifted Poisson random variables.
\end{abstract}
%

\received{\sday{16} \smonth{1} \syear{2018}}
\revised{\sday{17} \smonth{1} \syear{2018}}
\accepted{\sday{17} \smonth{1} \syear{2018}}
\publishedonline{\sday{31} \smonth{1} \syear{2018}}
\end{frontmatter}

\section{Introduction}

Survival analysis models time to an event of interest (e.g., lifetime).
It is a powerful tool in biometrics, epidemiology, engineering, and
credit risk assessment in financial institutions. The proportional
hazards model proposed in Cox (1972) \cite{cox1972} is a widely used
technique to characterize a relation between survival time and covariates.

Our model is presented in Augustin (2004) \cite{Augustin} where the
baseline hazard function $\lambda(\cdot)$ is assumed to belong to a
parametric space, while we consider $\lambda(\cdot)$ belonging to a
closed convex subset of $C[0,\tau]$. In practice covariates are often
contaminated by errors, so we deal with errors-in-variables model.
Kukush et al. (2011) \cite{KukBar} derive a simultaneous estimator of the
baseline hazard rate $\lambda(\cdot)$ and the regression parameter
$\beta$ and prove the consistency of the estimator. At that, the
parameter set $\varTheta_{\lambda}$ for the
baseline hazard rate is assumed to be bounded and separated away from
zero. The asymptotic normality of the estimator is shown in Chimisov
and Kukush (2014) \cite{ChiKu}.
In \cite{KuCh,KuChArh} we construct an estimator $ (\hat
{\lambda}^{(1)}_n(\cdot),~\hat{\beta}^{(1)}_n )$ of $\lambda(\cdot
)$ and $\beta$ over the parameter set $\varTheta=\varTheta_{\lambda}\times
\varTheta_{\beta}~$, where $n$ is the sample size and $\varTheta_{\lambda}$
is a subset of $C[0,\tau]$, which is unbounded from above and not
separated away from zero. The estimator is consistent and can be
modified to be asymptotically normal.

The goal of present paper is to construct confidence
intervals for integral functionals of $\lambda(\cdot)$ and a
confidence region for $\beta$ based on the estimators from \cite{KuCh,KuChArh}.
We impose certain restrictions on the error
distribution. Actually we handle three cases:
(a) the measurement error is bounded, (b) it is a normally distributed
random vector, and (c) it has independent components which are shifted
Poisson random variables.

The paper is organized as follows. Section~\ref{Model} describes the
observation model, gives main assumptions, defines an estimator under
an unbounded parameter set, and states the asymptotic normality result
from \cite{KuCh,KuChArh}.
Sections~\ref{ConRpar} and \ref{ConfBHR} present the main results: a
confidence region for the regression parameter and confidence intervals
for integral functionals of the baseline hazard rate.
Section~\ref{auxEst} provides a method to compute auxiliary consistent
estimates, and Section~\ref{Concl} concludes.

Throughout the paper, all vectors are column ones, $\mathsf{E}$ stands
for the expectation, $\mathsf{Var}$ stands for the variance, and
$\mathsf{Cov}$ for the covariance matrix. A relation holds \textit
{eventually} if it is valid for all sample sizes $n$ starting from some
random number, almost surely.

\section{The model and estimator}\label{Model}

Let $T$ denote the lifetime and have the intensity function
\begin{equation*}
\label{eqindm} \lambda(t|X;\lambda_0,\beta_0)=
\lambda_0(t)\exp\bigl(\beta_0^\top X\bigr),
\quad t\geq0.
\end{equation*}
A covariate $X$ is a time-independent random vector distributed in
$\mathbb{R}^m$, $\beta$ is a parameter belonging to $\varTheta_\beta\subset
\mathbb{R}^m$, and $\lambda(\cdot) \in\varTheta_\lambda\subset C[0,\tau]$
is a baseline hazard function.

We observe censored data, i.e., instead of $T$ only a censored lifetime
\\$Y:=\min\{T,C\}$ and the censorship indicator $\varDelta:=I_{\{T\leq C\}
}$ are available, where the censor $C$ is distributed on a given
interval $[0,\tau]$. The survival function of censor $G_C(u):=1-F_C(u)$
is unknown. The conditional pdf of $T$ given $X$ is
\begin{equation*}
\label{pdf T} f_T(t|X)=\lambda(t | X;\lambda_0,
\beta_0)\exp \Biggl(-\int_0^t
\lambda(t | X;\lambda_0, \beta_0) d s \Biggr).
\end{equation*}
The conditional survival function of $T$ given $X$ equals
\begin{equation*}
\label{surT} G_T(t|X)= \exp \Biggl(-\int_0^t
\lambda(s|X;\lambda_0,\beta_0) ds \Biggr) = \exp
\Biggl(-e^{\beta_0^\top X}\int_0^t
\lambda_0(s)ds \Biggr).
\end{equation*}

We deal with an additive error model, which means that instead of $X$,
a surrogate variable
\[
W=X+U
\]
is observed. We suppose that a random error $U$ has known moment
generating function ${M_U(z):=\mathsf{E} e^{z^\top U}}$, where $||z||$
is bounded according to assumptions stated below. A couple $(T, X)$,
censor $C$, and measurement error $U$ are stochastically independent.

Introduce assumptions from \cite{KuCh,KuChArh}.
\begin{enumerate}
\item[(i)]\label{i} $\varTheta_\lambda\subset C[0,\tau]$ is the following
closed convex set of nonnegative functions
\begin{align*}
\varTheta_\lambda:=\bigl\{& f:[0,\tau]\to\mathbb{R} |\; f(t)\geq0,\forall t
\in[0,\tau]\ \text{and}\\
 &\big|f(t)-f(s)\big|\leq L|t-s|,\ \forall t,s\in[0,\tau]\bigr\},
\end{align*}
where $L>0$ is a fixed constant.\label{c1}
\item[(ii)] $\varTheta_\beta\subset\mathbb{R}^m$ is a compact set.
\end{enumerate}
\begin{enumerate}
\item[(iii)] $\mathsf{E} U=0$ and for some fixed $\epsilon>0$,
\begin{equation*}
\mathsf{E} e^{2 D\| U \|}<\infty, \ \text{with} \ D:=\max\limits
_{\beta\in\varTheta_\beta}\| \beta\| +
\epsilon.
\end{equation*}

\item[(iv)] $\mathsf{E} e^{2 D \| X \|}< \infty$, where $D$ is defined
in (iii).
\item[(v)] $\tau$ is the right endpoint of the distribution of $C$,
that is\\
$\mathsf{P}(C>\tau)=0$ and for all $\epsilon>0$,
$\mathsf{P}(C>\tau-\epsilon)>0$.
\item[(vi)] The covariance matrix of random vector $X$ is positive definite.

Denote
\begin{equation}
\label{K} \varTheta=\varTheta_\lambda\times\varTheta_\beta.
\end{equation}

\item[(vii)] The couple of true parameters $(\lambda_0,\beta_0)$
belongs to $\varTheta$ given in (\ref{K}), and moreover ${\lambda
_0(t)>0}$, $t\in[0,\tau]$.

\item[(viii)] $\beta_0$ is an interior point of $\varTheta_\beta$.
\item[(ix)] $\lambda_0\in\varTheta_\lambda^\epsilon$ for some $\epsilon
>0$, with\\
\begin{align*}
\varTheta_\lambda^\epsilon:=\bigl\{& f:[0,\tau]\to\mathbb{R}~ |~ f(t)
\geq \epsilon,\; \forall t \in[0,\tau]\ \text{and}\\
& \big|f(t)-f(s)\big|\leq(L-\epsilon)|t-s|,\forall t,s\in[0,\tau] \bigr\}.
\end{align*}
\item[(x)]\label{x} $\mathsf{P}(C>0) =1$.
\end{enumerate}

Consider independent copies of the model $(X_i, T_i, C_i, Y_i, \varDelta
_i, U_i, W_i)$,\\ ${i=1,\ldots,n}$. Based on triples $(Y_i,\varDelta_i,W_i)$,
${i=1,\ldots,n}$, we estimate true parameters $\beta_0$ and $\lambda
_0(t)$, $t \in[0,\tau]$.
Following Augustin (2004) \cite{Augustin}, we use the corrected partial
log-likelihood function
\[
Q_n^{cor}(\lambda,\beta):=\frac{1}{n} \sum
_{i=1}^{n} q(Y_i,\varDelta_i,W_i;\lambda,\beta),
\]
with
\[
q(Y,\varDelta,W;\lambda,\beta):=\varDelta\cdot\bigl(\log\lambda(Y)+
\beta^\top W\bigr)-\frac{\exp(\beta^\top W)}{M_U(\beta)}\int_0^Y
\lambda(u)du.
\]

The estimator \cite{KuCh,KuChArh} of the baseline hazard rate
$\lambda(\cdot)$ and
parameter $\beta$ is defined as follows.
\begin{defin}\label{def1}
Fix a sequence $\{\varepsilon_n\}$ of positive numbers, with
$\varepsilon_n\downarrow0$, as $n\to\infty$. The corrected estimator
$ (\hat{\lambda}^{(1)}_n,\hat{\beta}^{(1)}_n )$ of $(\lambda
,\beta)$ is a Borel measurable function of observations $(Y_i,\varDelta
_i,W_i)$, $i=1,\ldots,n$, with values in $\varTheta$ and such that
\begin{equation}
\label{def} Q_n^{cor} \bigl(\hat{\lambda}^{(1)}_n,
\hat{\beta}^{(1)}_n \bigr)\geq \sup_{(\lambda,\beta)\in\varTheta}Q_n^{cor}(
\lambda,\beta)-\varepsilon_n.
\end{equation}
\end{defin}

Theorem 3 from \cite{KuCh,KuChArh} proves that under conditions
(i) to (vii) the corrected estimator $ (\hat{\lambda}^{(1)}_n,\hat
{\beta}^{(1)}_n )$ is a strongly consistent estimator of the true
parameters $(\lambda_0,\beta_0)$. In the proof of Theorem 3 from \cite
{KuCh,KuChArh}, it is shown that \textit{eventually} and for
$R$ large enough, the upper bound on the right-hand side of (\ref{def})
can be taken over the set $\varTheta^R:=\varTheta_{\lambda}^R\times\varTheta
_\beta$, with
\begin{equation*}
\varTheta_{\lambda}^R:=\varTheta_{\lambda}\cap
\bar{B}(0,R),
\end{equation*}
where $\bar{B}(0,R)$ denotes the closed ball in $C[0,\tau]$ with center
in the origin and radius $R$. Thus, we assume that for all $n\geq1$,
\begin{equation}
\label{defeqi} Q_n^{cor} \bigl(\hat{\lambda}^{(1)}_n,
\hat{\beta}^{(1)}_n \bigr)\geq \sup_{(\lambda,\beta)\in\varTheta^R}Q_n^{cor}(
\lambda,\beta)-\varepsilon_n
\end{equation}
and $(\hat{\lambda}^{(1)}_n,\hat{\beta}^{(1)}_n)\in\varTheta^R$. Notice
that $\varTheta^R$ is a compact set in $C[0,\tau]$.

Definition 2 from \cite{KuCh,KuChArh} provides, based on $
(\hat{\lambda}^{(1)}_n,\hat{\beta}^{(1)}_n )$,
a modified estimator $ (\hat{\lambda}^{(2)}_n,\hat{\beta
}^{(2)}_n )$ which is consistent and asymptotically normal.

\begin{defin} The modified corrected estimator $ (\hat{\lambda
}^{(2)}_n,\hat{\beta}^{(2)}_n )$ of $(\lambda,\beta)$ is a Borel
measurable function of observations $(Y_i,\varDelta_i,W_i)$,
${i=1,\ldots,n}$, with values in $\varTheta$ and such that
\begin{equation*}
\label{def3} \bigl(\hat{\lambda}^{(2)}_n,\hat{
\beta}^{(2)}_n \bigr)= %
\begin{cases} \arg\max\{Q_n^{cor}(\lambda,\beta)\ | \ (\lambda,\beta
)\in\varTheta,\;\mu_\lambda\geq\frac{1}{2}\mu_{\hat{\lambda}^{(1)}_n}\},
&\text{if~}\; \mu_{\hat{\lambda}^{(1)}_n}>0;\\
 (\hat{\lambda}^{(1)}_n,\hat{\beta}^{(1)}_n ),&  \text{otherwise},
\end{cases} %
\end{equation*}
where $\mu_\lambda:=\min_{t\in[0,\tau]} \lambda(t)$.
\end{defin}

Below we use notations from \cite{ChiKu}. Let
\begin{align*}
a(t)&=\mathsf{E}\bigl[Xe^{\beta_0^\top X}G_T(t|X)\bigr], \quad b(t)=
\mathsf {E}\bigl[e^{\beta_0^\top X}G_T(t|X)\bigr], \quad \varLambda(t)=
\int_0^t \lambda_0(t) dt,\\
p(t)&=\mathsf{E}\bigl[XX^\top e^{\beta_0^\top X}G_T(t|X)
\bigr], \quad T(t)=p(t)b(t)-a(t)a^\top(t), \quad K(t)=\frac{\lambda_0(t)}{b(t)},\\
A&=\mathsf{E} \Biggl[XX^\top e^{\beta_0^\top X} \int_0^Y
\lambda_0 (u) du \Biggr], \quad M=\int_0^{\tau}
T(u) K(u)G_c(u)du.
\end{align*}
For $i=1,2,\ldots$, introduce random variables
\[
\zeta_i=-\frac{\varDelta_i a(Y_i)}{b(Y_i)}+\frac{\exp(\beta_0^\top
W_i)}{M_U(\beta_0)}\int
_0^{Y_i}a(u)K(u)du+\frac{\partial q}{\partial
\beta}(Y_i,
\varDelta_i,W_i,\beta_0,\lambda_0),
\]
with
\[
\frac{\partial q}{\partial\beta}(Y,\varDelta,W;\lambda,\beta)=\varDelta \cdot W-\frac{M_U(\beta)W-\mathsf{E}(U e^{\beta^\top U})}{M_U(\beta)^2}
\exp\bigl(\beta^\top W\bigr)\int_0^Y
\lambda(u)du.
\]
Let
\begin{align*}
\varSigma_{\beta}&=4\cdot\mathsf{Cov}(\zeta_1),\quad m(
\varphi_\lambda)=\int_0^{\tau}
\varphi_\lambda(u)a(u)G_C(u)du,\\
\sigma^2_{\varphi}&=4\cdot\mathsf{Var}~ \bigl\langle
q'(Y,\varDelta,W,\lambda _0,\beta_0),\varphi
\bigr\rangle=4\cdot\mathsf{Var}~\xi (Y,\varDelta, W ),
\end{align*}
with
\begin{equation}
\label{xi} %
\begin{split} \xi (Y,\varDelta, W ) &=\frac{\varDelta\cdot\varphi_\lambda
(Y)}{\lambda_0(Y)}-
\frac{\exp(\beta_0^\top W)}{M_U(\beta_0)}\int_0^Y \varphi_\lambda(u)du
+\varDelta\cdot\varphi_{\beta}^\top W
\\
&\quad- \varphi_{\beta}^\top \frac{M_U(\beta_0)W- \mathsf{E}[U e^{\beta
^\top_0U}]}{M_U(\beta_0)^2}\exp\bigl(
\beta_0^\top W\bigr)\int_0^Y
\lambda_0(u)du, \end{split} %
\end{equation}
where $\varphi=(\varphi_{\lambda},\varphi_{\beta})\in C[0,\tau]\times
\mathbb{R}^m$
and $q'$ denotes the Fr\'echet derivative.

\begin{thm}[\cite{KuCh,KuChArh}]\label{th3} Assume conditions
(i) -- (x). Then $M$ is nonsingular and
\begin{equation}
\label{beta2} \sqrt{n}\bigl(\hat{\beta}^{(2)}_n-
\beta_0\bigr)\xrightarrow{\text {d}}N_m
\bigl(0,M^{-1}\varSigma_{\beta}M^{-1}\bigr).
\end{equation}
Moreover, for any Lipschitz continuous function $f$ on $[0,\tau]$,
\begin{equation*}
\sqrt{n}\int_0^\tau\bigl(\hat{
\lambda}^{(2)}_n-\lambda_0\bigr)
(u)f(u)G_C(u) du\xrightarrow{\text{d}}N\bigl(0,\sigma^2_{\varphi}(f)
\bigr),
\end{equation*}
where $\sigma^2_{\varphi}(f)=\sigma^2_{\varphi}$ with $\varphi=(\varphi
_{\lambda},\varphi_{\beta})$, $\varphi_{\beta}=-A^{-1}m(\varphi_{\lambda
})$ and $\varphi_\lambda$ is a unique solution in $C[0, \tau]$ to the
Fredholm integral equation
\begin{equation*}
\frac{\varphi_\lambda(u)}{K(u)}-a^\top(u)A^{-1}m(\varphi_{\lambda
})=f(u),
\quad u\in[0, \tau].
\end{equation*}
\end{thm}

\section{Confidence regions for the regression parameter}\label{ConRpar}

Denote as $\mathsf{E}_X[\cdot]$ the conditional expectation given a
random variable $X$. Remember that $M_U(z)=\mathsf{E} e^{z^\top U}$.
For simplicity of notation, we write $M_{k, \beta}$ instead of
$M_U((k+1)\beta)$.
Using differentiation in $z$ one can easily prove the following.
\begin{lem}\label{lema}
The equalities hold true:
\begin{align*}
e^{z^\top X}&=\frac{\mathsf{E}_X[e^{z^\top W}]}{M_U(z)},\\
Xe^{z^\top X}&=\frac{1}{M_U(z)} \biggl(\mathsf{E}_X
\bigl[We^{z^\top W}\bigr]-\frac
{\mathsf{E}[U e^{z^\top U}]}{M_U(z)}\mathsf{E}_X
\bigl[e^{z^\top W}\bigr] \biggr),\\
 XX^\top e^{z^\top X}&= \frac{1}{M_U(z)}
\biggl(\mathsf{E}_X\bigl[WW^\top e^{z^\top W}\bigr]-2
\frac{\mathsf
{E}[U e^{z^\top U}]}{M_U(z)}\mathsf{E}_X \bigl[W^\top
e^{z^\top W}\bigr] -
\\
&\quad - \biggl(\frac{\mathsf{E}[UU^\top e^{z^\top U}]}{M_U(z)}-2\frac
{\mathsf{E}[Ue^{z^\top U}]\cdot\mathsf{E}[U^\top e^{z^\top
U}]}{M_U^2(z)} \biggr)
\mathsf{E}_X \bigl[e^{z^\top W}\bigr] \biggr).
\end{align*}
\end{lem}

Now, we state conditions on measurement error $U$ under which one can
construct unbiased estimators for $a(t)$, $b(t)$ and $p(t)$, $t\in
[0,\tau]$.

\begin{thm}\label{soleq}
Suppose that for any $\beta\in\varTheta_{\beta}$ and $A>0$,
\begin{equation}
\label{maincond} \sum_{k=0}^\infty
\frac{a_{k+1}(\beta)}{k!}A^k<\infty,
\end{equation}

with
\[
a_{k+1}(\beta) := \frac{\mathsf{E}\|U\|^2 e^{(k+1)\beta^\top U}}{M_{k,
\beta}}~.
\]

Then there exist functions $B(\cdot,\cdot)$, $A(\cdot,\cdot)$ and
$P(\cdot,\cdot)$ which satisfy deconvolution equations:
\begin{enumerate}
\item[(a)] $\mathsf{E}_X [B(W,t)]=\exp (\beta^\top X-\varLambda
(t)e^{\beta^\top X} )$,
\item[(b)] $\mathsf{E}_X [A(W,t)]=X \exp (\beta^\top X-\varLambda
(t)e^{\beta^\top X} )$,

\item[(c)] $\mathsf{E}_X [P(W,t)]=XX^\top \exp (\beta^\top
X-\varLambda(t)e^{\beta^\top X} )$; ~$t\in[0,\tau]$.
\end{enumerate}
\end{thm}

\begin{proof}
We find solutions to the equations in a form of series expansions using
the idea from Stefanski (1990) \cite{Stef}.

(a) Utilizing Taylor decomposition of the right-hand side, we obtain
\[
\exp \bigl(\beta^\top X-\varLambda(t)e^{\beta^\top X} \bigr)= \sum
_{k=0}^\infty g_k(X,t), \quad
g_k(X,t):=\frac{(-1)^k}{k!}\varLambda ^k(t)e^{(k+1)\beta^\top X}.
\]

Using Lemma \ref{lema} take for $k\geq0$
\[
B_k(W,t)=\frac{(-1)^k}{k!M_{k, \beta}}\varLambda^k(t)e^{(k+1)\beta^\top W},
\]
so that
$\mathsf{E}_X [B_k(W,t)]=g_k(X,t)$, $t\in[0,\tau]$.
If we ensure that
\[
\sum\limits
_{k=0}^\infty\mathsf{E}_X
\big|B_k(W,t)\big|<\infty,
\]
then $B(W,t)=\sum_{k=0}^\infty B_k(W,t)$ is a solution to the
first equation. We have
\[
\sum_{k=0}^\infty\mathsf{E}_X
\big|B_k(W,t)\big|=\sum_{k=0}^\infty
\frac
{\varLambda^k(t)}{k!} e^{(k+1)\beta^\top X}=\exp \bigl(\beta^\top X+
\varLambda(t)e^{\beta^\top
X} \bigr)<\infty.
\]
Here no additional restriction on $U$ is needed.\vadjust{\goodbreak}

(b) Similarly, we show that $A(W,t)=\sum_{k=0}^\infty A_k(W,t)$, with
\[
A_k(W,t):=\frac{(-1)^k}{k!M_{k, \beta}}\varLambda^k(t) \biggl[ W -
\frac
{\mathsf{E} [U e^{(k+1)\beta^\top U}]}{M_{k, \beta}} \biggr] e^{(k+1)\beta^\top W},
\]
is a solution to the second equation, if $\sum_{k=0}^\infty
\mathsf{E}_X\|A_k(W,t)\|<\infty$.
We have
\begin{align*}
 &\sum_{k=0}^\infty
\mathsf{E}_X\|A_k(W,t)\|
\\
&\quad=\sum_{k=0}^\infty\frac{\varLambda^k(t)}{k!M_{k, \beta}} \mathsf
{E}_X \biggl\llvert \biggl\llvert X+U -\frac{\mathsf{E} [U e^{(k+1)\beta^\top U}]}{M_{k,
\beta}} \biggr
\rrvert \biggr\rrvert e^{(k+1)\beta^\top(X+U)}
\\
&\quad\leq\|X\| \exp \bigl(\beta^\top X+\varLambda(t)e^{\beta^\top X}
\bigr) +2\sum_{k=0}^\infty\frac{\varLambda^k(t)}{k!}
\frac{\mathsf{E}
\|U\|e^{(k+1)\beta^\top U}}{M_{k, \beta}}e^{(k+1)\beta^\top X}.
\end{align*}
The latter sum is finite due to condition (\ref{maincond}).
Therefore, there exists a solution to the second equation.

(c) Finally, for the third equation we put
%
\begin{align*}
&P_k(W,t)\\
&\quad=\frac{(-1)^k\varLambda^k(t)}{k!M_{k, \beta}}
\biggl[WW^\top e^{(k+1)\beta^\top W}-2\frac{\mathsf{E} [U e^{(k+1)\beta^\top
U}]}{M_{k, \beta}} W^\top
e^{(k+1)\beta^\top W}
\\
&\qquad - \biggl(\frac{\mathsf{E}[UU^\top e^{(k+1)\beta^\top U}]}{M_{k,
\beta}}-2\frac{\mathsf{E} [U e^{(k+1)\beta^\top U}]\cdot\mathsf
{E}[U^\top e^{(k+1)\beta^\top U}]}{M_{k, \beta}^2} \biggr)
e^{(k+1)\beta
^\top W} \biggr].
\end{align*}

The matrix $P(W,t)=\sum_{k=0}^\infty P_k(W,t)$ is a solution to
the third equation if
\begin{equation}
\label{finite} \sum\limits
_{k=0}^\infty\mathsf{E}_X\big\|P_k(W,t)\big\|<
\infty.
\end{equation}
Hereafter $\|Q\|$ is the Euclidean norm of a matrix $Q$. We have
\begin{equation}
\label{normP} %
\begin{split} \sum\limits
_{k=0}^\infty
\mathsf{E}_X\big\|P_k(W,t)\big\|&\leq\sum
\limits
_{k=0}^\infty \frac{\varLambda^k(t)}{k!} \biggl[\frac{\mathsf{E}_X
[~\|W\|^2 e^{(k+1)\beta^\top W}~]}{M_{k, \beta}}
\\
&\quad +2\frac{\mathsf{E} [~\|U\| e^{(k+1)\beta^\top U}~] \cdot\mathsf
{E}_X [~\|W\| e^{(k+1)\beta^\top W}~]}{M_{k, \beta}^2}
\\
&\quad+\frac{\mathsf{E}[~\|U\|^2 e^{(k+1)\beta^\top U}~]\cdot\mathsf
{E}_X e^{(k+1)\beta^\top W}}{M_{k, \beta}^2}
\\
&\quad +2\frac{ (\mathsf{E} \|U\| e^{(k+1)\beta^\top U}
)^2\cdot\mathsf{E}_X e^{(k+1)\beta^\top W}}{M_{k, \beta}^3} \biggr]. \end{split} %
\end{equation}
%
The right-hand side of (\ref{normP}) is a sum of four series which can
be bounded similarly based on condition (\ref{maincond}). E.g., for the
last of the four series we have:\vadjust{\goodbreak}
\begin{align*}
\bigl(\mathsf{E} \|U\| e^{\frac{1}{2}(k+1)\beta^\top U} e^{\frac
{1}{2}(k+1)\beta^\top U} \bigr)^2~
&\leq\mathsf{E}\|U\|^2 e^{(k+1)\beta
^\top U} \cdot M_{k, \beta}~,\\
\mathsf{E}_X e^{(k+1)\beta^\top W}&=M_{k, \beta}\cdot
e^{(k+1)\beta^\top X},\\
\sum\limits
_{k=0}^\infty \frac{\varLambda^k(t) (\mathsf{E} \|U\|
e^{(k+1)\beta^\top U} )^2\cdot\mathsf{E}_X e^{(k+1)\beta^\top
W}}{k!~M_{k, \beta}^3}
&\leq\sum\limits
_{k=0}^\infty \frac{a_{k+1}(\beta)\varLambda^k(t)
e^{(k+1)\beta^\top X} }{k!}<\infty.
\end{align*}

Therefore, condition (\ref{maincond}) yields (\ref{finite}), and
$P(W,t)$ is a solution to the third equation.
\end{proof}

\begin{thm} The condition of Theorem \ref{soleq} is fulfilled in each
of the following cases:

(a) the measurement error $U$ is bounded,

(b) $U$ is normally distributed with zero mean and variance-covariance
matrix $\sigma^2_U I_m$, with $\sigma_U>0$, and

(c) $U$ has independent components $U_{(i)}$ which are shifted Poisson
random variables, i.e. $U_{(i)}=\tilde{U}_{(i)}-\mu_i$, where $\tilde
{U}_{(i)}\sim Pois(\mu_i)$, $i=1,\ldots, m$.
\end{thm}

\begin{proof}

(a) Let $\|U\|\leq K$. Then
\[
\quad\frac{\mathsf{E} \|U\|^2 e^{(k+1)\beta^\top U}}{M_{k, \beta}} \leq K^2 ,
\]
and (\ref{maincond}) holds true.

(b) For a normally distributed vector $U$ with components $U_{(i)}$, we
have $\mathsf{E} e^{tU_{(i)}}=\exp (\frac{t^2\sigma^2_U}{2}
)$. Differentiation twice in $t$ gives
\[
\mathsf{E} U_{(i)}^2e^{(k+1)\beta_iU_{(i)}}= \bigl(1+(k+1)^2
\beta_i^2 \sigma^2_U \bigr)
\sigma^2_U \exp \biggl(\frac{(k+1)^2\beta_i^2\sigma
^2_U}{2} \biggr),
\]
and
\[
\frac{\mathsf{E} U_{(i)}^2e^{(k+1)\beta^\top U}}{M_{k, \beta}}= \bigl(1+(k+1)^2\beta_i^2
\sigma^2_U \bigr)\sigma^2_U.
\]
Thus,
\[
\frac{\mathsf{E} \|U\|^2 e^{(k+1)\beta^\top U}}{M_{k, \beta}}=\sum\limits
_{i=1}^m
\bigl(1+(k+1)^2\beta_i^2\sigma_U^2
\bigr)\sigma_U^2.
\]
Then (\ref{maincond}) holds true.

(c) We have $M_{U_{(i)}}(t):=\mathsf{E}e^{tU_{(i)}}=\exp(\mu
_i(e^t-1)-\mu_i t)$.
Differentiation twice in $t$ gives
\begin{align*}
M''_{U_{(i)}}(t)&=\mathsf{E}U_{(i)}^2e^{U_{(i)}t}=
\mu_i^2\bigl( e^t-1\bigr)^2M_{U_{(i)}}(t)+
\mu_i e^tM_{U_{(i)}}(t),\\
\frac{\mathsf{E} U_{(i)}^2e^{(k+1)\beta^\top U}}{M_{k, \beta}}&= \mu _i^2 \bigl(e^{(k+1)\beta_i}-1
\bigr)^2+\mu_i e^{(k+1)\beta_i}\le\textrm {const}\cdot
e^{2(k+1)\cdot|\beta_i|},
\end{align*}
where the factor `const' does not depend of $k$. Thus,
\[
\frac{\mathsf{E} \|U\|^2 e^{(k+1)\beta^\top U}}{M_{k, \beta}}\le\textrm {const}\cdot\sum\limits
_{i=1}^m
\ e^{2(k+1)\cdot|\beta_i|},
\]
and condition (\ref{maincond}) holds. This completes the proof.
\end{proof}

Now, we can construct estimators of $a(t)$, $b(t)$ and $p(t)$ for $t\in
[0,\tau]$.
Take $\hat{\varLambda}(t):=\int_0^t\hat{\lambda}^{(2)}_n(s) ds$ as a
consistent estimator of $\varLambda(t)$, $t\in[0,\tau]$. Indeed, the
consistency of $\hat{\lambda}_n^{(2)}(\cdot)$ implies
\[
\sup_{t\in[0,\tau]}\big|\hat{\varLambda}(t)-\varLambda(t)\big|\rightarrow0
\]
a.s. as $n\rightarrow\infty$.

For any fixed $(\lambda,\beta)\in\varTheta^R~$ and for all $t\in[0,\tau
]$, a sequence
\[
\frac{1}{n}\sum_{i=1}^n
B(W_i,t;\lambda,\beta)
\]
converges to $b(t;\lambda,\beta)$ a.s. due to SLLN. The sequence is
equicontinuous a.s. on the compact set $\varTheta^R$, and the limiting
function is continuous on $\varTheta^R$. The latter three statements
ensure that the sequence converges to $b$ uniformly on $\varTheta^R$. Thus,
\begin{equation*}
\hat{b}(t)=\frac{1}{n}\sum_{i=1}^n
B\bigl(W_i;\hat{\lambda}^{(2)}_n,\hat {
\beta}^{(2)}_n,\hat{\varLambda}\bigr)\rightarrow b(t;
\lambda_0,\beta_0,\varLambda ), \quad t\in[0,\tau],
\end{equation*}
a.s. as $ n\rightarrow\infty$.

In a similar way for all $t\in[0,\tau]$,
\begin{equation*}
\label{a} \hat{a}(t)=\frac{1}{n}\sum_{i=1}^n
A\bigl(W_i;\hat{\lambda}^{(2)}_n,\hat {
\beta}^{(2)}_n,\hat{\varLambda}\bigr)\rightarrow a(t;
\lambda_0,\beta_0,\varLambda)
\end{equation*}
a.s. and
\begin{equation*}
\label{p} \hat{p}(t)=\frac{1}{n}\sum_{i=1}^n
P\bigl(W_i;\hat{\lambda}^{(2)}_n,\hat {
\beta}^{(2)}_n,\hat{\varLambda}\bigr)\rightarrow p(t;
\lambda_0,\beta_0,\varLambda)
\end{equation*}
a.s.
Then
\[
\hat{T}(t)\hat{K}(t)= \biggl(\hat{p}(t)-\frac{\hat{a}(t)\hat{a}^\top
(t)}{\hat{b}(t)} \biggr)\hat{
\lambda}^{(2)}_n(t)
\]
is a consistent estimator of $T(t)K(t)$, $t\in[0,\tau]$.

\begin{defin}
The Kaplan--Meier estimator of the survival function of censor $C$ is
defined as
\[
\hat{G}_C(u)= %
\begin{cases} \prod\limits_{j=1}^{n} ( \frac{N(Y_j)}{N(Y_j)+1}
)^{\tilde{\varDelta}_j I_{Y_j\leq u}}& \textrm{if}\ u\leq Y_{(n)};\\
0, \;& \textrm{otherwise},
\end{cases} %
\]
where $\tilde{\varDelta}_j:=1-\varDelta_j$, $N(u):=\sharp\{Y_i>u,\;
i=1,\ldots,n\}$, and $Y_{(n)}$ is the largest order statistic.
\end{defin}

We state the convergence of the Kaplan--Meier estimator. Remember that
$Y=\min\{T,C\}$. Let $G_Y(t)$ be the survival function of $Y$.

\begin{thm}[\cite{Foldes}]\label{KM} Assume the following:
\begin{enumerate}
\item[(a)] survival functions $G_T$ and $G_C$ are continuous, and
\item[(b)] it holds
\[
\min\bigl\{G_Y(S),1-G_Y(S)\bigr\}\geq\delta,
\]
for some fixed $0<S<\infty$ and $0<\delta<\frac{1}{2}~$.
\end{enumerate}
Then a.s. for all $n\geq2$,
\begin{equation}
\label{convKM} \sup_{1\leq i\leq n, Y_i\leq S}\big|\hat{G}_n(Y_i)-G_C(Y_i)\big|=O
\biggl(\sqrt {\frac{\ln n}{n}} \biggr).
\end{equation}
\end{thm}

In our model, the lifetime $T$ has a continuous survival function, and
if we assume that the same holds true for the censor $C$, then the
first condition of Theorem \ref{KM} is satisfied. Next, it holds
$G_Y(t)=G_T(t)G_C(t)$ and due to condition (v) for all small enough
positive $\varepsilon$ there exists $0<\delta<\frac{1}{2}$ such that
\[
\delta\leq G_T(\tau-\varepsilon)G_C(\tau-\varepsilon)
\leq1-\delta.
\]
Therefore, the second condition holds as well, with $S=\tau-\varepsilon$.

Relation (\ref{convKM}) is equivalent to the following: there exists a
random variable $C_S(\omega)$ such that a.s. for all $n\geq2$,
\begin{equation*}
\sup_{ 0\leq u \leq S}\big|\hat{G}_{C}(u)-G_C(u)\big|\leq
C_S(\omega) \sqrt {\frac{\ln n}{n}}~.
\end{equation*}

Let 
\[
\hat{M}=\int_0^{Y_{(n)}} \hat{T}(u) \hat{K}(u)
\hat{G}_C(u) du.
\]
We have
%
\begin{align}
\label{M} \|\hat{M}-M\|&= \Biggl\llvert \Biggl\llvert \int
_0^{Y_{(n)}} \bigl(\hat{T}(u) \hat {K}(u)
\hat{G}_C(u)-T(u)K(u)G_C(u) \bigr)du+
\nonumber
\\
&\quad+ \int_{Y_{(n)}}^{\tau}T(u)K(u)G_C(u)
du \Biggr\rrvert \Biggr\rrvert
\nonumber
\\
&\leq\sup_{ 0\leq u \leq\tau}\big\|\hat{T}(u)\hat{K}(u)-T(u)K(u)\big\| \int
_0^{Y_{(n)}}\hat{G}_C(u) du
\nonumber
\\
&\quad+\int_0^{Y_{(n)}}\big\|T(u)K(u)\big\|\cdot\big|
\hat{G}_C(u)-G_C(u)\big|du
\nonumber
\\
&\quad+G_C(Y_{(n)})\int_{Y_{(n)}}^{\tau}\big\|T(u)K(u)\big\|
du.
\end{align}
%
Due to the above-stated consistency of $\hat{T}(\cdot)\hat{K}(\cdot)$
and since $\hat{G}_C $ is bounded by 1, the first summand in (\ref{M})
converges to zero a.s. as $n\to\infty$.

Consider the second summand. Let $S=\tau-\varepsilon$ for some fixed
$\varepsilon>0$. There are two possibilities:
$Y_{(n)}\leq S$ and $ S< Y_{(n)}\leq\tau$. In the first case,
\[
\int_0^{Y_{(n)}}\big\|T(u)K(u)\big\|\cdot\big|
\hat{G}_C(u)-G_C(u)\big|du\leq\textrm {const}\cdot\sup
\limits
_{ 0\leq u \leq S}\big|\hat{G}_C(u)-G_C(u)\big|.
\]
In the second case,
\begin{align*}
&\int_0^{Y_{(n)}}\big\|T(u)K(u)\big\|\cdot\big|
\hat{G}_C(u)-G_C(u)\big|du\\
&\quad\leq\textrm {const} \Bigl( \sup
\limits_{ 0\leq u \leq S}\big|
\hat{G}_C(u)-G_C(u)\big|
+\int_S^{Y_{(n)}}\big|\hat{G}_C(u)-G_C(u)\big|
du \Bigr)\\
&\quad \leq\textrm {const} \Bigl(\sup\limits
_{ 0\leq u \leq S}\big|\hat {G}_C(u)-G_C(u)\big|+Y_{(n)}-S
\Bigr).
\end{align*}

It holds that $Y_{(n)}\to\tau\quad\text{a.s. } $ Utilizing Theorem
\ref{KM}, we first tend $n\to\infty$ and then $\varepsilon\to0$ and
obtain convergence of the second summand of (\ref{M}) to 0 a.s. as $n\to
\infty$.

The convergence of $Y_{(n)}$ yields the convergence of the third
summand. Finally,
\[
\|\hat{M}-M\|\to0\quad\text{a.s. \;as} \quad n\to\infty.
\]

Because $\mathsf{E}\zeta_i=0$, it holds $\varSigma_{\beta}=4\cdot\mathsf
{E}\zeta_1\zeta_1^\top$.
Therefore, we take
\begin{align*}
\hat{\varSigma}_{\beta}&=\frac{4}{n}\sum
_{i=1}^n \hat{\zeta_i}\hat{\zeta
}_i^\top, \quad\text{with}\\
\hat{\zeta}_i&=-\frac{\varDelta_i \hat{a}(Y_i)}{\hat{b}(Y_i)}+\frac{\exp
(\hat{\beta}^{(2)T}_n W_i)}{M_U(\hat{\beta}^{(2)}_n)}\int
_0^{Y_i}\hat {a}(u)\hat{K}(u)du+\frac{\partial q}{\partial\beta}
\bigl(Y_i,\varDelta _i,W_i,\hat{
\beta}^{(2)}_n,\hat{\lambda}^{(2)}_n
\bigr),
\end{align*}
as an estimator of $\varSigma_{\beta}$.
We have
\[
\hat{\varSigma}_{\beta}\to\varSigma_{\beta} \quad\text{a.s. as}
\quad n\to \infty.
\]
Then
\begin{equation}
\label{estM} \hat{M}^{-1}\hat{\varSigma}_{\beta}
\hat{M}^{-1}\to M^{-1}\varSigma_{\beta}
M^{-1} \quad\text{a.s., }
\end{equation}
and \textit{eventually} $\hat{M}^{-1}\hat{\varSigma}_{\beta}\hat{M}^{-1}>0$.
Convergences (\ref{beta2}) and (\ref{estM}) yield
\[
\sqrt{n} \bigl(\hat{M}^{-1}\hat{\varSigma}_{\beta}
\hat{M}^{-1} \bigr)^{-1/2}\bigl(\hat{\beta}^{(2)}_n-
\beta_0\bigr)\xrightarrow{\text{d}}N(0,I_m).
\]
Thus,
\begin{align*}
 & \bigl\| \sqrt{n} \bigl(
\hat{M}^{-1}\hat{\varSigma}_{\beta}\hat {M}^{-1}
\bigr)^{-1/2}\bigl(\hat{\beta}^{(2)}_n-
\beta_0\bigr) \bigr\| ^2
\\
&\quad=n\bigl(\hat{\beta}^{(2)}_n-\beta_0
\bigr)^\top \bigl(\hat{M}^{-1}\hat{\varSigma }_{\beta}
\hat{M}^{-1} \bigr)^{-1}\bigl(\hat{\beta}^{(2)}_n-
\beta _0\bigr)\xrightarrow{\text{d}}\chi^2_m. %
\end{align*}

Given a confidence probability $1-\alpha$, the asymptotic confidence
ellipsoid for $\beta$ is the set
\begin{equation*}
E_n=\biggl\{z\in\mathbb{R}^m~\big|~\bigl(z-\hat{
\beta}^{(2)}_n\bigr)^\top\bigl(
\hat{M}^{-1}\hat {\varSigma}_{\beta}\hat{M}^{-1}
\bigr)^{-1}\bigl(z-\hat{\beta}^{(2)}_n\bigr)\leq
\frac
{1}{n}\bigl(\chi^2_m\bigr)_\alpha
\biggr\}.
\end{equation*}
Here $(\chi^2_{m})_\alpha$ is the upper quantile of $\chi^2_{m}$ distribution.

\section{Confidence intervals for the baseline hazard rate}\label{ConfBHR}
Theorem \ref{th3} implies the following statement.
\begin{cor}\label{coro} Let $0<\varepsilon<\tau$.
Assume that the censor $C$ has a bounded pdf on $[0, \tau- \varepsilon
]$. Under conditions (i) -- (x),
for any Lipschitz continuous function $f$ on $[0,\tau]$ with support on
$[0,\tau-\varepsilon]$,
\begin{equation*}
\label{lamb} \sqrt{n}\int_0^{\tau-\varepsilon}\bigl(\hat{
\lambda}^{(2)}_n-\lambda _0\bigr) (u)f(u)du
\xrightarrow{\text{d}}N\bigl(0,\sigma^2_{\varphi}(f)\bigr),
\end{equation*}
where $\sigma^2_{\varphi}(f)=\sigma^2_{\varphi}$ with $\varphi=(\varphi
_{\lambda},\varphi_{\beta})$, $\varphi_{\beta}=-A^{-1}m(\varphi_{\lambda
})$ and $\varphi_\lambda$ is a unique solution in $C[0, \tau]$ to the
Fredholm integral equation
\begin{equation}
\label{phieq} \frac{\varphi_\lambda(u)}{K(u)}-a^\top(u)A^{-1}m(
\varphi_{\lambda
})=\frac{f(u)}{G_C(u)},\quad u\in[0, \tau].
\end{equation}
Here we set $\frac{f(\tau)}{G_C(\tau)}=0$. Notice that $\frac{1}{G_C}$
is Lipschitz continuous on $[0,\tau-\varepsilon]$.
\end{cor}

We show that asymptotic variance $\sigma^2_{\varphi}$ is positive and
construct its consistent estimator.

\begin{defin}
A random variable $\xi$ is called nonatomic if
$\mathsf{P}(\xi=x_0)=0$, for all \(x_0 \in\mathbb{R}.\)
\end{defin}

\begin{lem}\label{lemVar}
Suppose that assumptions of Corollary \ref{coro} are satisfied.
Additionally assume the following:
\begin{enumerate}
\item[(xi)] $m(\varphi_{\lambda})\neq0$, for $\lambda=\lambda_0$ and
$\beta=\beta_0$.
\item[(xii)] For all nonzero $z\in\mathbb{R}^m$, at least one of random
variables $z^\top X$ and $z^\top U$ is nonatatomic.
\end{enumerate}
Then $\sigma^2_{\varphi}(f)\neq0$.
\end{lem}

\begin{proof}
We prove by contradiction. For brevity we drop zero index writing
$\varphi_\lambda=\varphi_{\lambda_0}$, $\varphi_\beta=\varphi_{\beta
_0}$ and omit arguments where there is no confusion. In particular, we
write $M_U$ instead of $M_U(\beta_0)$ and $\sigma^2_{\varphi}$ instead
of $\sigma^2_{\varphi}(f)$.

Denote
$\eta=\xi(C,0,W)$. From (\ref{xi}) we get
\[
M_U^2\cdot\eta=\int_0^C
\bigl(\alpha_W \varphi_{\lambda}(u)+\gamma _W
\lambda_0(u) \bigr)du,
\]
with
\[
\alpha_W:=-M_U\cdot\exp\bigl(\beta_0^\top
W\bigr),\quad\gamma_W:=-\varphi _{\beta}^\top
\bigl(M_U \cdot W- \mathsf{E}\bigl(U e^{\beta^\top
_0U}\bigr). \bigr)
\]

Suppose that $\sigma^2_{\varphi}=0$. This yields $\xi=0$ a.s. Then
\[
\eta=\xi\cdot I(\varDelta=0)=0 \quad \text{a.s}.
\]
It holds $\mathsf{P}(\varDelta=0)>0$ and according to (x), $C>0$ a.s.
Thus, in order to get a contradiction it is enough to prove that
\begin{equation}
\label{contr} \mathsf{P} (\eta=0~|~C>0 )=0.
\end{equation}
Since $C$ and $W$ are independent, it holds
\begin{equation*}
\label{expec} \mathsf{P} (\eta=0~|~C>0 )=\mathsf{E}[\pi_x|_{x=C}~|~C>0],
\end{equation*}
where for $x\in(0,\tau]$,
\begin{align*}
\pi_x:&=\mathsf{P} \Biggl( \int
_0^x \bigl(\alpha_W
\varphi_{\lambda
}(u)+\gamma_W\lambda_0(u)
\bigr)du=0 \Biggr)
\\
&=\mathsf{P} \Biggl( M_U \int_0^x
\varphi_{\lambda}(u)du +\varphi_{\beta
}^\top
\bigl(M_U \cdot W- \mathsf{E}\bigl(U e^{\beta^\top_0U}\bigr) \bigr)
\int_0^x \lambda_0(u) du=0
\Biggr)
\\
&=\mathsf{P} \bigl(~ \varphi_{\beta}^\top W =v_x
\bigr).
\end{align*}
Here $v_x$ is a nonrandom real number. In the latter equality we use
assumption (vii) to guarantee that $~~\int_0^x \lambda_0(u) du>0$.

Further, $\varphi_{\beta}=-A^{-1}m(\varphi_{\lambda})\neq0$
because according to (xi) $m(\varphi_{\lambda})\neq0$.
Using independence of $X$ and $U$ together with assumption (xii), we
conclude that
for all nonzero $z\in\mathbb{R}^m$, $z^\top W=z^\top X+z^\top U$ is
nonatomic. Then $\varphi_{\beta}^\top W$ is nonatomic as well and $\pi_x=0$.

Thus, $\mathsf{P} (\eta=0~|~C>0 )=0$ which proves (\ref
{contr}). Therefore, $\sigma^2_{\varphi}(f)\neq0$.\
\end{proof}

Now, we can construct an estimator for the asymptotic variance $\sigma
^2_{\varphi}~$. Rewrite
\[
A=\mathsf{E} \Biggl[XX^\top e^{\beta_0^\top X} \int_0^Y
\lambda_0 (u) du \Biggr]=\int_0^\tau
\lambda_0 (u) p(u) G_C(u) du.
\]
Let
\begin{equation*}
\hat{A}=\int_0^{Y_{(n)}} \hat{\lambda}^{(2)}_n
(u)\hat{p}(u)\hat {G}_C(u) du.
\end{equation*}
Results of Section~\ref{ConRpar} yield that $\hat{A}$ is a consistent
estimator of $A$. Denote
\[
\hat{m}(\varphi_\lambda)=\int_0^{Y(n)}
\varphi_\lambda(u)\hat{a}(u)\hat{G}_C(u)du
\]
and define $\hat{\varphi}_\lambda$ as a solution in $L_2[0,\tau]$ to
the Fredholm integral equation with a degenerate kernel
\begin{equation*}
\frac{\varphi_\lambda(u)}{\hat{K}(u)}-\hat{a}^\top\hat{T}(u)\hat {A}^{-1}
\hat{m}(\varphi_{\lambda})=\frac{f(u)}{\hat{G}_C(u)}~,\quad u\in [0, \tau].
\end{equation*}
\textit{Eventually}, a solution is unique because the limiting equation
(\ref{phieq}) has a unique solution. The function $\hat{\varphi}_\lambda
$ can be assumed right-continuous and it converges a.s. to $\varphi
_\lambda$ from (\ref{phieq}) in the supremum norm. Therefore,
\[
\hat{\varphi}_{\beta}=-\hat{A}^{-1}\hat{m}(\hat{
\varphi}_{\lambda})
\]
is a consistent estimator of $\varphi_\beta$.\vadjust{\goodbreak}

Finally, we construct an estimator of $\sigma_{\varphi}^2$. Put
\[
\hat{\sigma}^2_{\varphi}= \frac{4}{n-1}\sum
_{i=1}^{n} (\hat{\xi}_i-\bar{
\xi})^2,
\]
with
\begin{equation*}
\label{xiest} %
\begin{split} \hat{\xi}_i&:=
\frac{\varDelta_i \cdot\hat{\varphi}_\lambda(Y_i)}{\hat
{\lambda}^{(2)}_n(Y_i)}-\frac{\exp(\hat{\beta}^{(2)T}_n W_i)}{M_U(\hat
{\beta}^{(2)}_n)}\int_0^{Y_i}
\hat{\varphi}_\lambda(u)du +\varDelta_i \cdot\hat{
\varphi}_{\beta}^\top W_i
\\
&\ \quad- \hat{\varphi}_{\beta}^\top \frac{M_U(\hat{\beta}^{(2)}_n)W_i- \mathsf
{E}U e^{\hat{\beta}^{(2)T}_n U}}{M_U(\hat{\beta}^{(2)}_n)^2}\exp\bigl(\hat {
\beta}^{(2)T}_n W_i\bigr)\int
_0^{Y_i} \hat{\lambda}^{(2)}_n(u)du
\end{split} %
\end{equation*}
and
\[
\bar{\xi}:=\frac{1}{n}\sum_{i=1}^n
\hat{\xi}_i 
.
\]

Lemma \ref{lemVar} and the consistency of auxiliary estimators yield
the following consistency result.
\begin{thm}
Assume that condition (\ref{maincond})
together with conditions (i) -- (xii) are fulfilled and censor $C$ has
a continuous survival function. Then $\sigma^2_{\varphi}>0$ and
\begin{equation}
\label{sigmaest} \hat{\sigma}^2_{\varphi}\to\sigma^2_{\varphi}
\quad\text{a.s. as} \quad n\to\infty.
\end{equation}
\end{thm}

For fixed $\varepsilon>0$, consider an integral functional of the
baseline hazard rate, $I_f(\lambda_0)=\int_0^{\tau-\varepsilon} \lambda
_0(u) f(u) du$.
Corollary \ref{coro} gives
\begin{equation*}
\frac{\sqrt{n} (I_f(\hat{\lambda}^{(2)}_n)-I_f(\lambda_0))}{\sigma
_{\varphi}} \xrightarrow{\text{d}}N(0,1),
\end{equation*}
which together with (\ref{sigmaest}) yields
\begin{equation*}
\frac{\sqrt{n} (I_f(\hat{\lambda}^{(2)}_n)-I_f(\lambda_0))}{\hat{\sigma
}_{\varphi}} \xrightarrow{\text{d}}N(0,1).
\end{equation*}
Let
\begin{equation*}
I_n=\biggl[I_f\bigl(\hat{\lambda}^{(2)}_n
\bigr)-z_{\alpha/2}\frac{\hat{\sigma}_{\varphi
}}{\sqrt{n}}, I_f\bigl(\hat{
\lambda}^{(2)}_n\bigr)+z_{\alpha/2}\frac{\hat{\sigma
}_{\varphi}}{\sqrt{n}}
\biggr],
\end{equation*}
where $z_{\alpha/2}$ is the upper quantile of normal law. Then $I_n$ is
the asymptotic confidence interval for $I_f(\lambda_0)$.

\section{Computation of auxiliary estimators}\label{auxEst}

In Section~\ref{ConRpar}, we constructed estimators in a form of
absolutely convergent series expansions. E.g., in Theorem \ref{soleq}
(a) we derived an expansion of such kind for $t\in[0,\tau]$:
\[
B(W,t)=\sum_{k=0}^\infty B_k(W,t),
\quad\mathsf{E} B(W,t)=b(t)
\]
and
\[
\frac{1}{n}\sum_{i=1}^n
B(W_i,t)\rightarrow b(t),
\]
a.s. as $n\rightarrow\infty$. Now, we show that we can truncate the series.

Let $\{N_n: n\geq1\}$ be a strictly increasing sequence of nonrandom
positive integers. Fix $t$ for the moment and omit this argument $t$.
Consider the head of series $B(W_i)$,
\[
B_{N_i}(W_i):=\sum_{k=0}^{N_i}
B_k(W_i).
\]
Fix $j\geq1$, then for $n\geq j$ it holds:
\begin{align*}
\frac{1}{n}\sum_{i=j}^n
\big|B(W_i)-B_{N_i}(W_i)\big|&\leq\frac{1}{n}
\sum_{i=j}^n \sum
_{k=N_i+1}^\infty\big|B_k(W_i)\big|\\
&\leq
\frac{1}{n}\sum_{i=j}^n \sum
_{k=N_j+1}^\infty\big|B_k(W_i)\big|,\\
\limsup_{n\to\infty} \frac{1}{n}\sum
_{i=j}^n \big|B(W_i)-B_{N_i}(W_i)\big|
&\leq \limsup_{n\to\infty} \frac{1}{n}\sum
_{i=j}^n \sum_{k=N_j+1}^\infty\big|B_k(W_i)\big|\\
&=\mathsf{E}\sum_{k=N_j+1}^\infty\big|B_k(W_1)\big|.
\end{align*}
The latter expression tends to zero as $j\rightarrow\infty$.
Therefore, almost surely
\[
\lim_{j\to\infty}\limsup_{n\to\infty} \frac{1}{n}
\sum_{i=j}^n \big|B(W_i)-B_{N_i}(W_i)\big|=0.
\]
We conclude that
\[
\frac{1}{n}\sum_{i=1}^n
B_{N_i}(W_i)\rightarrow\mathsf{E} B(W_1)=b(t)
\]
a.s. as $n\rightarrow\infty$. Moreover, with probability one the
convergence is uniform in $(\lambda, \beta)$ belonging to a compact set.
Therefore, it is enough to truncate the series $B(W,t)$ by some large
numbers, which makes feasible the computation of estimators from
Section~\ref{ConRpar}.

\section{Conclusion}\label{Concl}
At the end of Section~\ref{ConRpar}, we constructed asymptotic
confidence intervals for integral functionals of the baseline hazard
rate $\lambda_0(\cdot)$, and at the end of Section~\ref{ConfBHR}, we
constructed an asymptotic confidence region for the regression
parameter $\beta$. We imposed some restrictions on the error
distribution. In particular, we handled the following cases:
(a) the measurement error is bounded, (b) it is normally distributed,
and (c) it has independent components which are shifted Poisson random
variables. Based on truncated series, we showed a way to compute
auxiliary estimates which are used in construction of the confidence sets.

In future we intend to elaborate a method to construct confidence
regions in case of heavy-tailed measurement errors.



\bibliographystyle{vmsta-mathphys}

\begin{thebibliography}{8}

\bibitem{Augustin}
%
\begin{barticle}
\bauthor{\bsnm{Augustin}, \binits{T.}}:
\batitle{An exact corrected log-likelihood function for {C}ox's proportional
hazards model under measurement error and some extensions}.
\bjtitle{Scand. J. Statist.}
\bvolume{31}(\bissue{1}),
\bfpage{43}--\blpage{50}
(\byear{2004}).
doi:\doiurl{10.1111/j.1467-9469.2004.00371.x}.
\MR{2042597}
\end{barticle}
%
%
\OrigBibText
%
\begin{barticle}
\bauthor{\bsnm{Augustin}, \binits{T.}}:
\batitle{An exact corrected log-likelihood function for {C}ox's proportional
hazards model under measurement error and some extensions}.
\bjtitle{Scand. J. Statist.}
\bvolume{31}(\bissue{1}),
\bfpage{43}--\blpage{50}
(\byear{2004}).
doi:\doiurl{10.1111/j.1467-9469.2004.00371.x}.
\MR{2042597}
\end{barticle}
%
\endOrigBibText
\bptok{structpyb}%
\endbibitem

\bibitem{ChiKu}
%
\begin{barticle}
\bauthor{\bsnm{Chimisov}, \binits{C.}},
\bauthor{\bsnm{Kukush}, \binits{A.}}:
\batitle{Asymptotic normality of corrected estimator in {C}ox proportional
hazards model with measurement error}.
\bjtitle{Mod. Stoch. Theory Appl.}
\bvolume{1}(\bissue{1}),
\bfpage{13}--\blpage{32}
(\byear{2014}).
doi:\doiurl{10.15559/vmsta-2014.1.1.3}.
\MR{3314791}
\end{barticle}
%
%
\OrigBibText
%
\begin{barticle}
\bauthor{\bsnm{Chimisov}, \binits{C.}},
\bauthor{\bsnm{Kukush}, \binits{A.}}:
\batitle{Asymptotic normality of corrected estimator in {C}ox proportional
hazards model with measurement error}.
\bjtitle{Mod. Stoch. Theory Appl.}
\bvolume{1}(\bissue{1}),
\bfpage{13}--\blpage{32}
(\byear{2014}).
doi:\doiurl{10.15559/vmsta-2014.1.1.3}.
\MR{3314791}
\end{barticle}
%
\endOrigBibText
\bptok{structpyb}%
\endbibitem

\bibitem{cox1972}
%
\begin{barticle}
\bauthor{\bsnm{Cox}, \binits{D.R.}}:
\batitle{Regression models and life tables (with discussion)}.
\bjtitle{Journal of the Royal Statistical Society, Series B}
\bvolume{34},
\bfpage{187}--\blpage{220}
(\byear{1972}).
\MR{0341758}
\end{barticle}
%
%
\OrigBibText
%
\begin{barticle}
\bauthor{\bsnm{Cox}, \binits{D.R.}}:
\batitle{Regression models and life tables (with discussion)}.
\bjtitle{Journal of the Royal Statistical Society, Series B}
\bvolume{34},
\bfpage{187}--\blpage{220}
(\byear{1972}).
\MR{0341758}
\end{barticle}
%
\endOrigBibText
\bptok{structpyb}%
\endbibitem

\bibitem{Foldes}
%
\begin{barticle}
\bauthor{\bsnm{F\"oldes}, \binits{A.}},
\bauthor{\bsnm{Rejt\"o}, \binits{L.}}:
\batitle{Strong uniform consistency for nonparametric survival curve estimators from
randomly censored data}.
\bjtitle{Ann. Statist.},
\bfpage{122}--\blpage{129}
(\byear{1981}).
\MR{0600537}
\end{barticle}
%
%
\OrigBibText
%
\begin{botherref}
\oauthor{\bsnm{F\"oldes}, \binits{A.}},
\oauthor{\bsnm{Rejt\"o}, \binits{L.}}:
Strong uniform consistency for nonparametric survival curve estimators from
randomly censored data.
Ann. Statist.,
122--129
(1981).
\MR{0600537}
\end{botherref}
%
\endOrigBibText
\bptok{structpyb}%
\endbibitem

\bibitem{KukBar}
%
\begin{barticle}
\bauthor{\bsnm{Kukush}, \binits{A.}},
\bauthor{\bsnm{Baran}, \binits{S.}},
\bauthor{\bsnm{Fazekas}, \binits{I.}},
\bauthor{\bsnm{Usoltseva}, \binits{E.}}:
\batitle{Simultaneous estimation of baseline hazard rate and regression
parameters in {C}ox proportional hazards model with measurement error}.
\bjtitle{J. Statist. Res.}
\bvolume{45}(\bissue{2}),
\bfpage{77}
(\byear{2011}).
\MR{2934363}
\end{barticle}
%
%
\OrigBibText
%
\begin{barticle}
\bauthor{\bsnm{Kukush}, \binits{A.}},
\bauthor{\bsnm{Baran}, \binits{S.}},
\bauthor{\bsnm{Fazekas}, \binits{I.}},
\bauthor{\bsnm{Usoltseva}, \binits{E.}}:
\batitle{Simultaneous estimation of baseline hazard rate and regression
parameters in {C}ox proportional hazards model with measurement error}.
\bjtitle{J. Statist. Res.}
\bvolume{45}(\bissue{2}),
\bfpage{77}
(\byear{2011}).
\MR{2934363}
\end{barticle}
%
\endOrigBibText
\bptok{structpyb}%
\endbibitem

\bibitem{KuChArh}
%
\begin{botherref}
\oauthor{\bsnm{Kukush}, \binits{A.}},
\oauthor{\bsnm{Chernova}, \binits{O.}}:
Consistent estimation in {C}ox proportional hazards model with measurement
errors and unbounded parameter set.
arXiv preprint arXiv:1703.10940
(2017).
\end{botherref}
%
\bid{mr={3666874}}
%
\OrigBibText
%
\begin{botherref}
\oauthor{\bsnm{Kukush}, \binits{A.}},
\oauthor{\bsnm{Chernova}, \binits{O.}}:
Consistent estimation in {C}ox proportional hazards model with measurement
errors and unbounded parameter set.
arXiv preprint arXiv:1703.10940
(2017)
\end{botherref}
%
\endOrigBibText
\bptok{structpyb}%
\endbibitem

\bibitem{KuCh}
%
\begin{barticle}
\bauthor{\bsnm{Kukush}, \binits{A.}},
\bauthor{\bsnm{Chernova}, \binits{O.}}:
\batitle{Consistent estimation in {C}ox proportional hazards model with
measurement errors and unbounded parameter set ({U}krainian)}.
\bjtitle{Teor. Imovir. Mat. Stat.}
\bvolume{96},
\bfpage{100}--\blpage{109}
(\byear{2017}).
\MR{3666874}
\end{barticle}
%
%
\OrigBibText
%
\begin{barticle}
\bauthor{\bsnm{Kukush}, \binits{A.}},
\bauthor{\bsnm{Chernova}, \binits{O.}}:
\batitle{Consistent estimation in {C}ox proportional hazards model with
measurement errors and unbounded parameter set ({U}krainian)}.
\bjtitle{Teor. Imovir. Mat. Stat.}
\bvolume{96},
\bfpage{100}--\blpage{109}
(\byear{2017}).
\MR{3666874}
\end{barticle}
%
\endOrigBibText
\bptok{structpyb}%
\endbibitem

\bibitem{Stef}
%
\begin{barticle}
\bauthor{\bsnm{Stefanski}, \binits{L.A.}}:
\batitle{Unbiased estimation of a nonlinear function of a normal mean with
application to measurement error models}.
\bjtitle{Comm. Statist. Theory Methods}
\bvolume{18}(\bissue{12}),
\bfpage{4335}--\blpage{4358}
(\byear{1990}).
\MR{1046712}
\end{barticle}
%
%
\OrigBibText
%
\begin{barticle}
\bauthor{\bsnm{Stefanski}, \binits{L.A.}}:
\batitle{Unbiased estimation of a nonlinear function of a normal mean with
application to measurement error models}.
\bjtitle{Comm. Statist. Theory Methods}
\bvolume{18}(\bissue{12}),
\bfpage{4335}--\blpage{4358}
(\byear{1990}).
\MR{1046712}
\end{barticle}
%
\endOrigBibText
\bptok{structpyb}%
\endbibitem

\end{thebibliography}

\end{document}